\documentclass{amsart}

\usepackage{ifthen}
\usepackage{graphicx}
\usepackage{amssymb}
\usepackage{enumerate}
\newtheorem{anyprop}{Anyprop}[section]

\newtheorem{theorem}[anyprop]{Theorem}
\newtheorem{lemma}[anyprop]{Lemma}
\newtheorem{proposition}[anyprop]{Proposition}
\newtheorem{corollary}[anyprop]{Corollary}

\theoremstyle{definition}

\newtheorem{conjecture}[anyprop]{Conjecture}

\newtheorem{example}[anyprop]{Example}

\newtheorem{remark}[anyprop]{Remark}

\newtheorem{definitionintro}{Definition}

\newcommand  {\fop}     {\mathfrak{p}}

\newcommand  {\codim}   {\operatorname{codim}}

\renewcommand  {\ker }  {\operatorname{ker}}

\newcommand  {\rad}      {\operatorname{rad}}

\newcommand  {\Spec}    {\operatorname{Spec}}

\theoremstyle{remark}

\numberwithin{equation}{section}

\usepackage{amscd}
\usepackage{amssymb}
\input xy
\xyoption{all}

\newcommand{\idealq}{{\mathfrak q}}

\newcommand{\idealp}{{\mathfrak p}}

\begin{document}
\title[NORMALITY AND RELATED PROPERTIES OF FORCING ALGEBRAS]
{NORMALITY AND RELATED PROPERTIES OF FORCING ALGEBRAS}

\author[Danny de Jes\'us G\'omez-Ram\'irez]{Danny de Jes\'us G\'omez-Ram\'irez}
\author[Holger Brenner]{Holger Brenner}

\address{Vienna University of Technology, Institute of Discrete Mathematics and Geometry,
wiedner Hauptstaße 8-10, 1040, Vienna, Austria.}
\address{Universit\"at Osnabr\"uck, Fachbereich 6: Mathematik/Informatik,
Albrechtstr. 28a,
49076 Osnabr\"uck, Germany}
\email{daj.gomezramirez@gmail.com}
\email{hbrenner@uni-osnabrueck.de}

%\thanks{to my parents}


\begin{abstract}
We present a sufficient condition for irreducibility of forcing algebras and study the (non)-reducedness phenomenon. Furthermore, we prove a criterion for normality for forcing algebras over a polynomial base ring with coefficients in a perfect field. This gives a geometrical normality criterion for algebraic (forcing) varieties over algebraically closed fields. Besides, we examine in detail an specific (enlightening) example with several forcing equations. Finally, we compute explicitly the normalization of a particular forcing algebra by means of finding explicitly the generators of the ideal defining it as an affine ring.
\end{abstract}
\maketitle
%fragt Holger

\smallskip

\footnote{This paper should be cited as follows D. A. J. G\'omez-Ram\'irez and H. Brenner. Normality and Related Properties of Forcing Algebras. Communications in Algebra. Volume 44, Issue 11, pp. 4769-4793, 2016.}

\section*{Introduction}
Let $R$ be a commutative ring, $I=(f_1,\ldots ,f_n)$ a finitely generated ideal and $f$ an arbitrary element of
$R$. A very natural and important question, not only from the theoretical but also from the computational point
of view, is to determine if $f$ belongs to the ideal $I$ or to some ideal closure of it
(for example to the radical, the integral closure, the plus closure, the solid closure, the tight closure, among others).
To answer this question the concept of a forcing algebra introduced by Mel Hochster in the context of
solid closure \cite{hochstersolid} is important (for more information on forcing algebras see \cite{brennerforcingalgebra},
\cite{brenneraffine}):

\begin{definitionintro} 
Let $R$ be a commutative ring, $I=(f_1, \ldots , f_n)$ an ideal and $f \in R$ another element. Then
the \emph{forcing algebra} of these (forcing) data is 
\[  A=R[T_1, \ldots, T_n] /(f_1T_1 + \ldots + f_nT_n+f ) \,  . \]
\end{definitionintro}

Intuitively, when we divide by the forcing equation $f_1T_1 + \ldots + f_nT_n+f $ we are ``forcing'' the element
$f$ to belong to the expansion of $I$ in $A$. Besides, it has the universal property that for any $R$-algebra $S$
such that $f\in IS$, there exists a (non-unique) homomorphism of $R$-algebras $\theta: A\rightarrow S$.

Furthermore, the formation of forcing algebras commutes with arbitrary change of base. Formally, if
$\alpha: R\rightarrow S$ is a homomorphism of rings, then 
\[S\otimes_R A\cong S[T_1,...,T_n]/(\alpha_1(f_1)T_1+\cdots+\alpha_n(f_n)T_n+\alpha(f))\]
is the forcing algebra for the forcing data $\alpha(f_1),\ldots ,\alpha(f_n),\alpha(f)$ . In particular,
if $\idealp \in X =\Spec R$, then the fiber of (the forcing morphism) $\varphi: Y:=\Spec A \rightarrow X:=\Spec R$ over $\idealp $, $\varphi^{-1}( \idealp)$, is the scheme theoretical
fiber $\Spec (\kappa (\idealp)\otimes_R A)$, where $\kappa (\idealp)=R_\idealp/ \idealp R_\idealp$ is its residue field.
In this case, the fiber ring $\kappa (\idealp) \otimes_R A$ is the forcing algebra over $\kappa (\idealp)$
corresponding to the forcing data $f_1(\idealp),\ldots ,f_n(\idealp),f(\idealp)$, where we
denote by $g(\idealp)\in \kappa (\idealp)$, the image (the evaluation) of $g\in R$ inside the residue field $\kappa (\idealp)=R_{\idealp}/\idealp R_{\idealp}$. 
Also, note that for any $f_i$ $A_{f_i}\cong R_{f_i}[T_1,...,\check{T_i},...,T_n]$, via the $R_{f_i}-$homomorphism sending
 $T_i\mapsto -\sum_{j\neq i}(f_j/f_i)T_j-(f/f_i)$ and $T_r\mapsto T_r$ for $r\neq i$.

An extreme case occurs when the forcing data consists only of $f$. Then, we define $I$ as the zero ideal. 
Therefore $A=R/(f)$. 

Besides, if $n=1$, then intuitively the forcing algebra $A=R[T_1]/(f_1T_1-f)$ can be consider as the graphic of the ``rational'' 
function $f/f_1$. We will explore this example in more detail in chapter two.

By means of forcing algebras and forcing morphisms one can rewrite the fact that the element $f$ belongs to a particular closure operations of $I$.
 We shall illustrate this now.

Firstly, the fact that $f\in I$ is equivalent to the existence of a homomorphism of $R-$algebras $\alpha:A\rightarrow R$, which is equivalent
 at the same time to the existence of a section $s:X\rightarrow Y$, i.e. $\varphi\circ s=Id_X$.

Secondly, $f$ belongs to the radical of $I$ if and only if $\varphi$ is surjective. In fact, suppose that $\varphi$ is surjective and let us
 fix a prime ideal $\fop\in X$ containing $I$. Then, $\varphi^{-1}(x)={\rm Spec}\,\kappa(\fop)\otimes A\neq \emptyset$, that means,
 $\kappa(\fop)\otimes A=\kappa(\fop)[T_1,...,T_n]/(f_1(\fop)T_1+\cdots+f_n(\fop)T_n+f(\fop))\neq 0$. But, each $f_i(\fop)=0$, since 
$f_i\in \fop$, therefore $f(\fop)$ is also zero, thus $f\in \fop$. In conclusion, $f\in \cap_{\fop\in V(I)}\fop= \rad I$. Conversely,
 suppose that $f\in \rad I$ and take an arbitrary prime $\fop\in X$. Then, if $I$ is not contained in $\fop$, then some $f_j(\fop)\neq0$ 
and so $\kappa(\fop)\otimes A\neq0$, that means $\varphi^{-1}(\fop)\neq\emptyset$. Lastly, if $I\subseteq\fop$ then $f\in \fop$, and therefore 
$\kappa(\fop)\otimes A=\kappa(\fop)[T_1,...,T_n]\neq0$ and thus $\varphi^{-1}(\fop)=\mathbb{A}_{\kappa(\fop)}^{n-1}\neq\emptyset$. In 
conclusion $\varphi$ is surjective.

Thirdly, let us review the definition of the tight closure of an ideal $I$ of a commutative ring $R$ of characteristic $p>0$. We say that 
$u\in R$ belongs to the \emph{tight closure} of $I$, denoted by $I^*$, if there exists a $c\in R$ not in any minimal prime, such that for all $q=p^e\gg0$,
 $cu^q\in I^{[q]}$, where $I^{[q]}$ denotes the expansion of $I$ under the $e-$th iterated composition of the Frobenius homomorphism 
$F:R\rightarrow R$, sending $x\rightarrow x^p$. Tight Closure is one of the most important closure operations in commutative algebra and 
was introduced in the 80s by M. Hochster and C. Huneke as an attempt to prove the ``Homological Conjectures'' (for more information 
\cite{Hochsterhuneketightclosure}). Let $(R,m)$ be normal local domain of dimension two. Suppose that $I=(f_1,...,f_n)$ is an $m-$primary 
ideal and $f$ is an arbitrary element of $R$. Then, $f\in I^*$ if and only if $D(IA)=\Spec\,A\smallsetminus V(IA)$ is not an affine scheme, i.e.
is not of the form $\Spec\,D$ for any commutative ring $D$ (see \cite[corollary 5.4.]{brenneraffine}).

Forth, the origin of the forcing algebras comes from the definition of the solid closure, as an effort to defining a closure operation 
for any commutative ring, independent the characteristic (see \cite{hochstersolid}). Explicitly, let $R$ be a Noetherian ring, let $I\subseteq R$ an ideal and
 $f\in R$. Then, $f$ belongs to the \emph{solid closure} of $I$ if for any maximal ideal $m$ of $R$ and any minimal ideal $\idealq$ of its 
completion $\widehat{R}_m$, for the complete local domain $(R'=\widehat{R}_m/\idealq,m')$ holds that the $d-$th local cohomology of the 
forcing algebra $A'$, obtained after the change of base $R\hookrightarrow R'$, $H_m^d(A')\neq0$, where $d=\dim R'$ 
(see \cite[Definition 2.4., p. 15]{brennerbarcelona}).

Fifth, let us consider an integral domain $R$ and an ideal $I\subseteq R$. Then, $u$ belongs to the \emph{plus closure} of $I$, denoted by $I^+$, 
if there exists a finite extension of domains $R\hookrightarrow S$, such that $f\in IS$. If $R$ is a Noetherian domain and $I=(f_1,...f_n)\subseteq R$
 is an ideal and $f\in R$, then $f\in I^+$ if and only if there exists an irreducible closed subscheme $\widetilde{Y}\subseteq Y=\Spec\,A$
 such that $\dim\widetilde{Y}=\dim X$, $\varphi(\widetilde{Y})=X$ and for each $x\in X$, $\varphi^{-1}(x)\cap\widetilde{Y}$ is finite 
(for an projective version of this criterion see \cite[Proposition 3.12]{brennerbarcelona}).

Finally, if $R$ denotes an arbitrary commutative ring and $I\subseteq R$ is an ideal, then we say that $u$ belongs to the \emph{integral 
closure} of $I$, denote by $\overline{I}$, if there exist $n\in \mathbb{N}$, and $a_i\in I^i$, for $i=1,...,n$, with

\[u^n+a_1u^{n-1}+\cdots+a_n=0.\]

We proved in \cite[Chapter 2]{brennergomezconnected} that $f\in \overline{I}$, where $I=(f_1,...,f_n)\subseteq R$, if and only if the 
corresponding forcing morphism $\varphi$ is universally connected, i.e. $\Spec (S\otimes_RA)$ is a connected space for any Noetherian 
change of base $R\rightarrow S$, such that $\Spec\,S$ is connected.

From this we derive a criterion of integrity for fractions $r/s\in K(R)$, where $R$ denotes a Noetherian domain, in terms of the universal 
connectedness of the natural forcing algebra $A:=R[T]/(sT+r)$. 

In view of this results, it seems very natural to study in commutative algebra the question of finding a closure operation with ``good'' 
properties (see \cite{epsteinclosureguide}), in terms of finding suitables algebraic-geometrical as well as topological or homological 
properties of the forcing morphism. This approach goes closer to the philosophy of Grothendieck's EGA of defining and studying the 
objects in a relative context (see \cite{hartshornealgebraic} and \cite{EGAI}). A simple and deep example of this approach is the counterexample
 to one of the most basic and important open questions on tight closure: the Localization Problem i.e., the question whether tight closure commutes with 
localization. This was done by H. Brenner and P. Monsky using vector bundles techniques and geometric deformations of tight closure 
(see \cite{brennerbarcelona}).

Besides, another good example going in this direction is a general definition of forcing morphism for arbitrary schemes. Specifically, 
let $X$ and $Y$ be arbitrary schemes. Suppose that $i:Z\rightarrow X$ is a closed subscheme and $f\in \Gamma(X,{\mathcal O}_X)$ is a global
 section. Then, a morphism $\varphi:Y\rightarrow X$ is a \emph{forcing morphism} for $f$ and $Z$, if 

\emph{i)} the pull-back of the restriction
 of $f$ to $Z$, $f_{|Z}=i_Z^{\sharp}(f)$ is zero, i.e.  $\varphi_{|\varphi^{-1}(Z)}^{\sharp}(f_{|Z})=0$; 

\emph{ii)} for any morphism 
of schemes $\psi: W\rightarrow X$ with the same property, i.e. $\psi_{|\psi^{-1}(Z)}^{\sharp}(f_{|Z})=0$, there exists a (non-unique) 
morphism $\widetilde{\psi}:W\rightarrow Y$ such that $\psi=\varphi\circ\widetilde{\psi}$. 
It is a natural generalization of the universal 
property of a forcing algebra but in the relative context and in a category including that of commutative rings with unity.
\newline\indent
In general, for a integral base ring there are two kinds of irreducible components for the prime spectrum of a forcing algebra. In fact, let $R$ be a noetherian domain, $I=(f_1 , \ldots ,f_n) $ an ideal, $f \in R$ and
$A=R[T_1, \ldots , T_n]/(f_1T_1+ \ldots +f_nT_n+f)$
the forcing algebra for these data. 
For $I \neq 0$ there exists a unique irreducible component $H \subseteq \Spec A$ (``horizontal component") with the property of dominating the
base $\Spec R$ (i.e. the image of $H$ is dense).
This component is given (inside $R[ T_1, \ldots , T_n]$) by
\[\idealp = R[ T_1, \ldots , T_n] \cap (f_1T_1+ \ldots + f_nT_n+f) Q(R)[ T_1, \ldots , T_n] \, ,\]
where $Q(R)$ denotes the quotient field of $R$.

All other irreducible components of $\Spec A$ (``vertical components") are of the form 
\[V( \idealq  R[ T_1, \ldots , T_n] )\]
for some prime ideal $\idealq \subseteq R$ which is minimal over $(f_1, \ldots , f_n,f)$ (for a complete proof of this fact and more information see \cite[Lemma 2.1.]{brennergomezconnected}).

Finally, let us describe briefly the content of the following sections of this paper.

We study the case corresponding to a sub-module $N$ of a finitely generated
module $M$ and an arbitrary element $s\in M$. This case corresponds to forcing algebras with several forcing equations

\[A=R[T_1,\ldots ,T_n]/\left\langle \left(\begin{array}{ccc} f_{11}& \ldots &f_{1n} \\ \vdots& \ddots & \vdots \\ f_{m1} & \ldots & f_{mn} \end{array}  \right) \cdot \left( \begin{array}{c}T_1 \\ \vdots \\ T_n \end{array}\right)+\left( \begin{array}{c}f_1\\ \vdots \\ f_m\end{array}
\right)   \right\rangle. \]

Even very basic properties of forcing algebras are not yet understood, and these paper deal in some extent with these questions.

For example, we describe how to perform elementary row and column operations on the forcing algebra by means of considering 
elementary affine linear isomorphisms and a specific relation between the regular sequences of forcing elements and the fitting ideals of 
the corresponding forcing matrix (\S1). 

Besides, the irreducibility of the forcing algebra over a noetherian domain can be obtained just by assuming that the height of $I$ is bigger or equal 
 that 2 (\S2).
 
Now, we show with two kinds of examples that for the reducedness of the forcing algebra it is not enough to have the reducedness of the base. Besides,
as a natural consequence of studying this, we see that a noetherian ring is the product of fields 
if and only if any element belongs to the ideal generated by its square power (\S3). Moreover if we add to the condition the possibility 
 that $I$ is the whole base $R$, then we get a complete characterization of the integrity of the forcing algebra over UFDs (\S4).
 
Moreover, with a very natural approximation through simple examples and increasing just step by step the dimension of the base space we obtain, in the
 case that our base is the ring of polynomials over a perfect field, a quite 
simple criterion of normality for the forcing algebras by means of the size of the codimensions of the ideal $I$ and the ideal $I+D$, where 
$D$ is generated by the partial derivatives of the data. In the case that we are working over an algebraic closed field and our base is 
the ring of coordinates of an irreducible variety $X$, the normality of the (forcing) hyperplane defined by the forcing equation can be 
characterized by imposing the condition that the codimension of the singular locus of $X$ in the whole affine space is a least three (\S5). Here, it is worth 
to know that we present the formal proof of this criterion as well as the ``informal'' way in which this criterion was originally found i.e., 
a way of analyzing simple examples increasing gradually the generality of the variables describing them, in order to develop slowly a deeper 
intuition of the phenomenon involved.

As an instance of the importance of the examples we analyze an specific forcing algebra, that we call the ``enlightening'' example, because
it is a very natural recurring point to verify the different results that we have already studied. On this respect this example is not less 
important that the former results. Instead of that, it is another valuable result where the different propositions and theorems 
come together (\S6). 

Moreover, we compute explicitly the normalization of a forcing algebra coming from the examples guiding us to find the normality criterion.
And again, on that process we deal with very elementary and fundamental questions related with normal domains and denominators ideals (\S7).

%put this in the introduction

\section{Forcing Algebras with several Forcing Equations}
Now, we study just a few elementary properties of forcing algebras which are defined by several forcing equations and which leads us in a natural 
way to the understanding of the linear algebra over the base ring $R$. This section could be understood as a simple invitation to this 
barely explored field of mathematics. Here we recommend for further reading \cite{brennerforcingalgebra}. 
In this case we can write the forcing algebra in a matrix form:
\[A=R[T_1,\ldots ,T_n]/\left\langle \left(\begin{array}{ccc} f_{11}& \ldots &f_{1n} \\ \vdots& \ddots & \vdots \\ f_{m1} & \ldots & f_{mn} \end{array}  \right) \cdot \left( \begin{array}{c}T_1 \\ \vdots \\ T_n \end{array}\right)+\left( \begin{array}{c}f_1\\ \vdots \\ f_m\end{array}
\right)   \right\rangle. \]

This corresponds to a submodule $N \subseteq M $ of finitely generated $R$-modules and an element $f \in M$ via a
free representation of these data (see \cite[p. 3]{brennerforcingalgebra}).

Now, we study how  the forcing algebra behaves when we make elementary row or column operations in the associated matrix $M$. 
Remember that the matrix notation in the forcing algebra just means that we are considering the ideal generated by the rows of the
resulting matrix, after performing the matrix multiplications and additions.

 First, if $l_1,...,l_m$ denote the rows of $M$, and $c\in R$ denote an arbitrary constant, making a row operation, $l_j\mapsto cl_j+l_i$,  
($i\neq j$; that is changing the $jth$ row by $c$ times 
the $ith$ row  plus the $jth$ row) just means changing the generators $h_1,...,h_m$ to the new generators 
$h_1,...,h_{j-1},ch_i+h_j,h_{j+1},...,h_m$. The ideal generated by these two groups of forcing elements coincides and therefore the 
associated forcing algebra are the same. Similarly, if we make operations of the form $l_i\mapsto l_j$ and $l_i\mapsto cl_i$,
 where $c$ is an invertible element of $R$, which correspond to change two rows and to multiply a row by an element in $R$, then the 
forcing algebra does not change.

For the column operation, the problem is a little bit more subtle. Let $\left\lbrace C_1,...,C_n\right\rbrace $ be the columns of 
the matrix $A$. Consider the column operation $\mapsto dC_i+C_jC_j$, where $d\in R$. Now, define the following automorphism 
$\varphi$ of the ring of polynomials $R[T_1,...,T_n]$ sending $T_s\mapsto T_s$, for $s\neq i$, and $T_i\mapsto cT_j+T_i$. 
Now, 
\[\varphi(h_r)=f_{r1}T_1+...+f_{ri}(cT_j+T_i)+...+f_{rn}T_n=\]
 \[f_{r1}T_1+...+(cf_{ri}+f_{rj})T_j+...+f_{rn}T_n.\] 
and then $\varphi$ induces an isomorphism between the forcing algebra with matrix $M$ and the forcing algebra with matrix obtained 
from $M$ performing the previous column operation. Similarly, for operations of the form $C_i\mapsto C_j$ and $C_i\mapsto dC_i$, 
where $d\in R$ is an invertible element, the resulting forcing algebras coincide. 
Now, if $R$ is a field and the rank of the associated matrix $M$ is $r$, where $r\leq {\rm min}(m,n)$, then performing row and column 
operations on the associated matrix we can obtained a matrix form by the $r\times r$ identity matrix in the upper-left side and 
with zeros elsewhere. 

Therefore, the elements $h_i$ have just the following simple form: $h_i=T_i+g_i$, for $i=1,...,r$ and $h_i=g_i$, 
for $i>r$, and some $g_i\in R$ (this $g_i$ could appear just in the nonhomogeneous case, corresponding to the changes made on the 
independent vector form by the $f_j$). Thus the forcing algebra $A$ is isomorphic either to zero (in the case that there exists $g_i\neq0$, 
for some $i>r$) or to $k[T_{r+1},...,T_n]$. This allow us to present the following lemma describing the fibers of a forcing algebra 
as affine spaces over the base residue field.

\begin{lemma}
\label{forcingfiber}
 Let $R$ be a commutative ring and let $A$ be the forcing algebra corresponding to the data
$\left\lbrace f_{ij}, f_i \right \rbrace$. Let $\idealp \in X$ be an arbitrary prime ideal of $R$ and $r$ the rank of the
matrix $\left\lbrace f_{ij}(\idealp) \right\rbrace$. Then the fiber over $\idealp $ is empty or isomorphic to the
affine space ${ \mathbb A}^{n-r}_{\kappa (\idealp)}$.
\end{lemma}
\begin{proof}
We know by a previous comment of the introduction that the fiber ring over $\idealp$ is $\kappa (\idealp)\otimes_RA$ which is just 
\[\kappa (\idealp)[T_1,\ldots ,T_n]/\left\langle \left(\begin{array}{ccc} f_{11}(\idealp)&\ldots &f_{1n}(\idealp)\\ \vdots& \ddots &\vdots \\f_{m1}(\idealp)&\ldots &f_{mn}(\idealp)\end{array}  \right) \cdot \left( \begin{array}{c}T_1 \\ \vdots \\ T_n \end{array}\right)+\left( \begin{array}{c}f_1(\idealp)\\ \vdots \\ f_m(\idealp)\end{array}
\right)  \right\rangle \, .\]
Now, making elementary row and column operations on the matrix $(f_{ij}(\idealp))$, as indicated before, we can obtain a matrix with 
zero entries except for the first $r$ entries of the principal diagonal which are ones, plus an independent vector.

In conclusion, after performing all the necessarily elementary operations, we obtain an isomorphism from $A$ to a very simple forcing 
algebra
 \[B=\kappa(\idealp)[T_1,...,T_n]/(T_1+g_1,...,T_r+g_r,g_{r+1},...,g_n),\]
 corresponding to the matrix with zero entries except for the first $r$ entries of the principal diagonal, which are ones. 
But then $B$ is clearly isomorph to the affine ring 
$\kappa(\idealp)[T_{r+1},...,T_n]$, if $g_{r+1}=\cdots=g_n=0$, and $A=0$ otherwise, proving our lemma.
\end{proof}

If $\kappa(\idealp)$ is algebraically closed, then the fiber over a point $\idealp \in \Spec R$ of this forcing algebra is just the 
solution set of the corresponding system of inhomogeneous linear equations over $\kappa (\idealp)$. If the 
vector $(f_1, \ldots , f_m)$ is zero, then we are dealing with a ``homogeneous'' forcing algebra. In this case there is a
(zero- or ``horizontal'') section $s: X =\Spec R \rightarrow Y = \Spec A$ coming from the homomorphism of
$R$-algebras from $A$ to $R$ sending each $T_i$ to zero. This section sends a prime ideal $\idealp \in X$ to the
prime ideal $(T_1,\ldots ,T_n)+\idealp \in Y$.
\begin{remark}
 If all $f_k$ are zero, and $m=n$, then the ideal $\mathfrak{a}$ is defined by the linear forms $h_i=f_{i1}T_1+\cdots +f_{in}T_n$, and in this case we can ``translate'' the fact of 
multiplying by the adjoint matrix of $M$, denoted by ${\rm adj}M$, just to saying that the elements $\det MT_i\in \mathfrak{a}$. In fact, 

\[\left( \begin{array}{c}\det M T_1\\\vdots\\ \det M T_n \end{array}\right)=\det M\cdot I_{nn}\cdot\left( \begin{array}{c}T_1\\\vdots\\T_n\end{array}\right) = {\rm adj}M\cdot M\cdot\left( \begin{array}{c}T_1\\\vdots\\T_n\end{array}\right)={\rm adj}M\cdot \left( \begin{array}{c}h_1\\\vdots\\h_n\end{array}\right)
 \]
where the entries of the last vector belong to $\mathfrak{a}$. From this fact we deduce that, when the determinant of $M$ is a unit in $R$, then
$\mathfrak{a}=(T_1,...,T_n)$ and the forcing algebra is isomorphic to the base ring $R$. Note that the previous argument works also in the 
nonhomogeneous case.
\end{remark}
Now, we study the homogeneous case when the elements $\left\lbrace h_1,...,h_m\right\rbrace$ form a regular sequence. 
First, we need the following general fact about the pure codimension of regular sequences in Noetherian rings: Let $S$ be a Noetherian ring and $\left\lbrace r_1,...,r_m\right\rbrace \subseteq S$ a regular sequence and $I$ the ideal generated 
by these elements. Then the pure codimension of $I$ is $m$. For a proof see \cite{gomezramirezthesis}.

Besides, if $j\in \left\lbrace 1,...,\min(m,n) \right\rbrace$, then we define the Fitting ideals $I_j$ as the ideals generated by
 the minors of size $j$ of the matrix $M$. This definition corresponds to the standard definition of Fitting ideals regarding $M$ 
as a $R$-homomorphism of free modules (see \cite[p. 497]{eisenbud}). 

\begin{proposition}
 Let $R$ be a Noetherian integral domain and $A$ the homogeneous forcing algebra corresponding to the data 
$\left\lbrace f_{ij}\right\rbrace$, with $i=1,...,n$ and $j=1,...,m$. Suppose that the forcing equations 
$\left\lbrace h_1,...,h_m \right\rbrace$ form a regular sequence in $B:=R[T_1,...,T_n]$. Then $m\geq n$ and $I_{{\rm min}(m,n)}\neq (0).$ 
\end{proposition}
\begin{proof}
 First, note that the ideal $I$ generated by the forcing elements is contained in the homogeneous ideal $P=(T_1,...,T_n)$, therefore 
we see that the dimension of $A$ is smaller or equal to the dimension of $B/P$, which is exactly the dimension of $R$. On the other 
hand, if we consider a saturated chain of primes in $A$, \[P_0\nsubseteq P_1\nsubseteq...\nsubseteq P_{\dim A},\] where $P_0$ is a 
minimal prime over $I$, then the former comments ${\rm ht}(P_0)=m$ and thus completing the former chain with a saturated 
chain for $P_0$ in $B$ of length $m$, we see that $\dim B\geq m+\dim A.$
Now, noting that $\dim B=\dim A+n$, since $A$ is Noetherian, we get $n+\dim A\geq m+\dim A$, which implies that $n \geq m$.

For the second part, let's consider the matrix $M$ in the field of fractions $K$ of $R$. It is an elementary fact that the rank of 
$M$ is $\leq s$ if and only if every minor of size $s+1$ of $M$ is zero. This follows from the fact that performing row operations on 
a matrix change the values of fixed minor of size $r$ of the original matrix just by a nonzero constant term of another minor of 
size $r$ of the changed matrix (this is a general way of saying that performing a row operation is just multiplying by an invertible 
matrix and therefore the fact that the determinant of the matrix is zero or not is independent of the row operation). 

Now, suppose 
by contradiction that $I_{{\rm min}(m,n)}=0$, then the rank of $M$ in $K$ is strictly smaller than ${\rm min}(m,n)$ and thus the rows of $M$ are 
linearly dependent in $K$. Without loss of generality, assume that there is $j\in \left\lbrace 1,...,{\rm min}(m,n)\right\rbrace$, such 
that the $jth$ row of $M$, $l_j$, is a linear combination of the former ones, that is, there exist 
some $\alpha_i\in K$, such that $l_j=\sum_{i=1}^{j-1}\alpha_il_i$. Now, after multiplying by a nonzero common multiple 
$\beta \in R$ of the denominators, we get an equation of the form $\beta l_j=\sum_{i=1}^{j-1}\gamma_il_i$, for some $\gamma_i \in R$.
  But, seeing this equality in $A$, (which means just multiplying this equality by the $n\times1$ vector given by the 
$T_i$) we see that $\beta h_j=\sum_{i=1}^{j-1}\gamma_i h_i$, which implies that $h_j$ is a zero divisor in $B/(h_1,...,h_{j-1})$, 
because $\beta \notin (T_1,...,T_n)$ ($\beta$ is a nonzero constant polynomial in $B$) and therefore $\beta \notin (h_1,...,h_{j-1})$. This contradicts the fact that $I$ is generated by a regular sequence.
\end{proof}
The converse of the previous proposition is false as the following example shows.
\begin{example}
Consider $R=k[x];B=R[T_1,T_2];h_1=xT_1-xT_2$ and $h_2=xT_1+xT_2$, where $k$ is a field of ${\rm char}k\neq 2$. Then $m=n=2$ and the 
determinant of the associated matrix is $2x^2$, but the sequence $\left\lbrace h_1,h_2\right\rbrace$ is not regular, in fact, 
the ideal $I$ generated by its elements has height just one, because it is contained in the principal ideal $(x)$, and therefore 
by former comments, $I$ cannot be generated by a regular sequence. Geometrically, the variety defined by $I$ 
is the union of a line $V(T_1,T_2)$ and a plane $V(x)$.

Intuitively, this example comes from the following observation. Suppose that we have the forcing algebra with equations $h_1'=T_1+T_2$ and 
$h_2=T_1-T_2$. If we consider the line $V(T_1-T_2,T_1+T_2)=V(T_1,T_2)$ 
(whose associated determinant is $2\neq0$) and multiplying these equations by $x$, we obtain a variety that is automatically 
the union of this line with the plane $V(x)\subseteq k^3$, which has bigger dimension, but the associated determinant of the new variety (our former 
example) 
is just $x^2$ times the former determinant. This process gives us a new variety with nonzero determinant but with an ideal with 
smaller codimension.
\end{example}

\section{Irreducibility}
Here we shall see that if $A$ is a forcing algebra over a Noetherian integral domain such that ${\rm ht}(f,f_1,...,f_n)\geq 2$, where 
$\{f_1,...,f_n,f\}$ is the forcing data, then $A$ is an irreducible ring (i.e. $A$ has just one minimal prime).

\begin{theorem}\label{irreducibility}
 Let $R$ be a Noetherian integral domain;
 \[A=R[T_1,...,T_n]/(f_1T_1+\cdots+f_nT_n+f);\]
 $h=f_1T_1+\cdots+f_nT_n+f$, where $f_1,...,f_n,f\in R$
 and $J=(f,f_1,...,f_n)$. Assume that ${\rm ht}J\geq2$, then $A$ is an irreducible ring. 
\end{theorem}
\begin{proof}
By Lemma \cite[Lemma 2.1.(2)]{brennergomezconnected}, it is enough to see that for any minimal prime $\idealq\in R$ of $J$, $\idealq B$ is not minimal over
$(h)$, because on that case, $A$ has just the horizontal component, and therefore is irreducible. 

Let $\idealq\in R$ be minimal over $J$. Then, 
\[{\rm ht}\idealq B\geq {\rm ht}\idealq\geq{\rm ht}J\geq2.\]
Therefore $\idealq B$ is not minimal over $(h)$, since by Krull's Principal Ideal Theorem the minimal primes over a principal ideal have 
height smaller or equal than one.
\end{proof}

\section{(Non)Reduceness}
In this section we study  the (non)reducedness of forcing algebras over a reduced base ring $R$. First, for a base field $k$ \cite[Lemma 3.1]{brennergomezconnected}
 shows that any forcing algebra is isomorphic to a ring of polynomials over $k$ or the zero algebra, therefore it is reduced. 
\newline\indent
Now, if $R$ is a local ring, let us first stay an elementary remark concerning a generalization of the Monomial Conjecture (MC) (see \cite{hochstercanonical}) in dimension
 one. 
\newline\indent

In dimension one (CM) just says that if $x\in m$ does not belong to any minimal prime ideal of $R$ then $x^n\notin (x^{n+1})$ for all
nonnegative integer $n$. In the next remark we will prove a generalization of this fact for a quasi-local ring, that is, not necessarily Noetherian.
\begin{remark}\label{mcindimone}
 Let $(R,m)$ be a quasi-local ring and $x\in m$. Then, there exists a positive integer $n$ such that $x^n\notin (x^{n+1})$, if and only if $x$ is
  a nilpotent or a unit.
  
  In fact, one direction is trivial, for the other one,
assume that $x$ is neither nilpotent nor a unit and that there exists $n\in \mathbb{N}$ and $y\in R$ such that $x^n=yx^{n+1}$,
thus $x^n(1-yx)=0$, but $1-yx\notin m$, therefore it is a unit, then $x^n=0$, which is a contradiction.
\end{remark}

\begin{example}

Let $(R,m)$ be a quasi-local reduced ring, which is not a field, and $f\in m\smallsetminus \{0\}$. Then, the trivial forcing algebra $A:=R/(f^2)$
is non-reduced because clearly $\overline{f}\in {\rm nil}A$ and by the previous Remark $\overline{f}\neq A$. So, there are always non-reduced
 forcing algebras over quasi-local reduced base rings, which are not a field.
 
\end{example}
Now, we want to study in which generality we can guarantee the existence of an element $f\in R$ such that $f\notin(f^2)$. 
The following Proposition gives a compact characterization of the fact that any element $f\in R$ belongs to $(f^2)$.

\begin{proposition}
A commutative ring with unity $R$ is reduced of dimension zero if and only if for any element $f\in R$, holds that $f\in (f^2)$.
In particular, if $R$ is noetherian then it is equivalent to the fact that $R$ is a finite direct product of fields. 
\end{proposition}
\begin{proof}
Assume that $R$ is a reduced zero-dimensional ring. Then it is enough to check the desired property locally. But, on that case $R$ 
is a ring with a unique prime ideal, which is at the same time reduced, therefore it is a field and in particular $f\in (f^2)$, for all 
$f\in R$.

For the other direction, let $P$ be a prime ideal of $R$. Clearly, the same property holds for $R/P$, thus, for any $g\in R/P$ there 
exists $c\in R/P$, such that $g=cg^2$. Therefore, $g(1-cg)=0$, implying that either $g=0$ or $1-cg=0$, because $R$ is an integral domain.
So $R/P$ is a field and then $R$ has dimension zero. Finally, from the hypothesis follows that for any $f\in R$, $f\in (f^{2^m})$, for 
every natural number $m$. In particular, if $f$ is nilpotent and $f^m=0$, for some $m\in\mathbb{N}$, then $f\in (f^m=(0))$. In conclusion,
$f$ is reduced. 
The second part is a direct consequence of the Chinesse Remainder Theorem.
 \end{proof}
 \begin{remark}
 The previous proposition guarantees the existence of non-reduced forcing algebras over any noetherian ring which is not a finite direct product of
  fields. Specifically, as before, we choose an element $f\in R$, such that $f\notin (f^2)$ and define $A:=R/(f^2)$.
\end{remark}
Finally, we present a more interesting example of an irreducible but not reduced forcing algebra over an affine domain base ring $R$ such that the 
$\codim((f_1,...,f_n),R)$ is arbitrary large.

\begin{example}
 Consider $R=k[x_1,...,x_{n-1},z]/(x_1z,...,x_{n+1}z)$, $h=x_1T_1+\cdots+x_{n+1}T_{n+1}+z^2$ and $A=R[T_1,...,T_{n+1})/(h)$.
 Then, 
 \[\codim((x_1,...,x_{n+1}),R)=n,\]
 because the ring of polynomials is catenary. Besides, it is straightforward to verify that 
 $z\notin (h)$, and $z^3=z^2h\in(h)$. Therefore $A$ is non-reduced.
 \end{example}

\section{Integrity over UFD}
Now, we prove an integrity criterion
 for forcing algebras over UFD as base ring involving just the height of the forcing elements $f_1,...,f_n$. 

\begin{lemma} \label{domain} Let $R$ be a Noetherian UFD which is not a field, $J=(f_1,\ldots,f_n,f)$, 
where some $f_i\neq0$, and  let $A$ be the forcing algebra corresponding to this data and $B=R[T_1,...,T_n]$. Then $A$ is an integral domain if and only if 
$J=R$, or ${\rm ht}J\geq2$.
  \end{lemma}
\begin{proof}
Along the proof we will use the basic fact that in a UFD the notions of prime and irreducible element coincide. We will prove the 
negation of the equivalence $((h)\in \Spec B) \Leftrightarrow (I=R \vee{\rm ht}J\geq2$), which is equivalent formally to 
 $((h)\notin \Spec B) \Leftrightarrow (I\neq R \wedge{\rm ht}J\leq1$). Now, we can written the condition at the right side by 
 ${\rm ht}J\leq1$, assuming implicitly that ${\rm ht} I$ is well defined, i.e., $I\neq R$. So we will see that $A$
is not an integral domain if and only if ${\rm ht}J\leq1$. In fact, we can assume $J\neq0$ and therefore ${\rm ht}J=1$. Choose a prime ideal $P$ of 
$R$ such that $P$ contains $J$ and ${\rm ht}P=1$. Choose $a\neq0\in P$. Now, some of the prime factors of $a$, say $p$, belongs to $P$ and therefore $P=(p)$, due to the fact that both 
prime ideals have height one. Thus, there exist $g_i,g\in R$ such that $f_i=pg_i$ and $f=pg$, hence 
$h=f_1T_1+ \ldots +f_nT_n+f=p(g_1T_1+\ldots +g_nT_n+g)$ is the product of $p$ and an element which is not a unit since some of the $f_i$ 
is different from zero. Therefore $h$ is not irreducible, which is equivalent of being a non prime element. In conclusion, $A$ is not an 
integral domain. 

Conversely, assume that $A$ is not an integral domain, or equivalently that $h=f_1T_1+ \ldots +f_nT_n+f$ is not irreducible. 
Hence, there exists polynomials $Q_1,Q_2\in R[T_1,\ldots ,T_n]$, not units, such that $h=Q_1Q_2$. Now, the degree of $h$ is the sum of 
the degrees of $Q_1$ and $Q_2$, because $R$ is an integral domain. Then one of the two factors has degree zero, say $Q_1$. Comparing the 
coefficients we get that each $f_i=Q_1g_i$ and $f=Q_1g$, and $Q_2=g_1T_1+\ldots+g_nT_n+g$. In conclusion, $J\subseteq(Q_1)R$ and 
therefore by Krull's Theorem ${\rm ht}(J)\leq1$.
\end{proof}

\section{A Normality Criterion for Polynomials over a Perfect Field}

Now we will try to understand under what conditions on the elements $f_1,\ldots ,f_n,f\in R$ the associated forcing algebra is a 
normal domain in the case that $R$ is the ring of polynomials over a perfect field. 
For some examples, results and intuition we assume a very basic and modest knowledge of algebraic 
 geometry, mainly relating affine varieties (see, for example \cite{goertzwedhornalgebraic} and \cite[Chapter I]{hartshornealgebraic}).

 We state explicitly the statement of a corollary of the Jacobian Criterion, which we use later. 
 For proofs see \cite[Theorem 16.19, Corollary 16.20]{eisenbud}.

\begin{corollary}\label{corjacobian}
Let $R[x_1,\ldots,x_r]/I$ be an affine ring over a perfect field $k$ and suppose that $I$ has pure codimension $c$, i.e., the height 
of any minimal prime over $I$ is exactly $c$.
 Suppose that $I=(f_1,\ldots,f_n)$. If $J$ is the ideal of $R$ generated by the $c\times c$ minors of the Jacobian matrix 
$(\partial f_i/\partial x_j)$, then $J$ defines the singular locus of $R$ in the sense that a prime $P$ of $R$  contains $J$  
if and only if $R_P$ is not a regular local ring.

\end{corollary}
Besides, let us recall Serre's Criterion for normality for any Noetherian ring (see \cite[Theorem 11.2.]{eisenbud}). Remember that a 
ring is normal if it is the direct product of normal domains:

\begin{theorem}
 A Noetherian ring $S$ is normal if and only if the following two conditions holds:
\begin{enumerate}
 \item (S2) For any prime ideal $P$ of $S$ holds \[ {\rm depth}_P(S_P)\geq {\rm min}(2,{\rm dim}(S_P)).  \]
\item (R1) Every localization of $S$ on primes of codimension at most one is a regular ring.
\end{enumerate}

\end{theorem}

 \begin{remark}\label{derivatives}
 If $R=k[x_1,\ldots ,x_r]$ and $h=f_1T_1 + \ldots + f_nT_n+f\in B:=R[T_1,\ldots ,T_n]$, $h\neq0$, then the forcing algebra 
 $A=R[T_1,\ldots ,T_n]/(h)$ is equidimensional of dimension ${\rm dim}A=r+n-{\rm ht}((h))=r+n-1$, since $=R[T_1,\ldots ,T_n]$ is catenary and $h$ has pure codimension 
 one, because every minimal prime over $(h)$ has height one by Krull's principal ideal 
 theorem. Therefore in the case that $k$ is a perfect field we deduce from the corollary of the Jacobian criterion that the singular 
 locus of the forcing algebra is exactly the prime spectrum of the following ring 
 \[ A_S=A/((\partial h/\partial x_j),(\partial h/ \partial T_i))=R[T_1,\ldots ,T_n]/(h,(\partial h/\partial x_j),(\partial h/ \partial T_i)).\]
Now, $(\partial h/\partial x_j)=\sum_{i=1}^n(\partial f_i/\partial x_j)T_i +(\partial f/\partial x_j)$ and
$\partial h/ \partial T_i=f_i$. Thus we get \[ J:=(h,(\partial h/\partial x_j),(\partial h/ \partial T_i))=(h,\sum_{i=1}^n(\partial f_i/\partial x_j)T_i+(\partial f/\partial x_j),f_i) \] 
\[=(f,f_i,\sum_{i=1}^n(\partial f_i/\partial x_j)T_i +(\partial f/\partial x_j)),\]
where $i,j\in \{1,...,n\}$. We can write the last set of generators in a compact way using matrices:

\[\left( \begin{array}{ccc}\partial f_1/\partial x_1&\ldots& \partial f_n/\partial x_1\\ \vdots&&\vdots \\ \partial f_1/\partial x_r &\ldots&\partial f_n/\partial x_r\end{array}\right)\cdot\left( \begin{array}{c}T_1\\ \vdots \\T_n\end{array}\right)+\left( \begin{array}{c}\partial f/\partial x_1 \\ \vdots \\ \partial f/\partial x_r \end{array}\right).\]
We will denote by $\overline{J}$ the class of $J$ in $A$.

\end{remark}
Now we rewrite the normality condition for the forcing algebra $A$ in terms of the codimension of its singular locus $V(\overline{J})\in \Spec A$, or 
in terms of the codimension of the corresponding closed subset $V(J)\subseteq \Spec (R[T_1,\ldots ,T_n])$, which are isomorphic as affine schemes.
On this section we set 
\[I=(f,f_1,...,f_n)\in R\]
and $D=(\partial f/\partial x_i, \partial f_j/\partial x_i)$ for $i,j\in \{1,...,n\}$.
 Note that $J\subseteq (I+D)B$. In particular. $V(IB)\cap V(DB)\subseteq V(J)\subseteq\Spec B$.

First, let's consider the trivial case $R=k$. By previous comments we know that if $A\neq0$ then $A=k[T_1,...,\check{T_i},...,T_n]$,
so $A$ is regular and thus a normal domain. In conclusion, $A$ is a normal domain if and only if all $f_i$ and $f$ are zero, or 
there exists some $f_i\neq0$.

\begin{lemma} \label{sing}
 let $R=k[x_1,\ldots ,x_r]$ be the ring of polynomials over a perfect field $k$ and $h=f_1T_1 + \ldots + f_nT_n+f\in B:=R[T_1,\ldots ,T_n]$, with $h\neq0$,
 and $A=R[T_1,\ldots ,T_n]/(h)$.
 Then, the following conditions are equivalent:
\begin{enumerate}
 \item $A$ is a normal ring.
\item ${\rm codim}(\overline{J},A)\geq2$, or  $\overline{J}=A$.
\item ${\rm codim}(J,B)\geq3$, or $J=B$.
\end{enumerate}
\end{lemma}
\begin{proof}
 $(1)\Rightarrow(2)$ Assume that $A$ is a normal ring, then the Serre's Criterion tells us that for any prime ideal $q$ of $A$ with 
 ${\rm ht}q\leq1$, $A_q$ is a regular ring (remember that in dimension zero regularity is equivalent to being a field). Now, suppose that $\overline{J}\subsetneq A$. Then, we know that for any prime $P$ of $A$ that 
 contains $\overline{J}$, $A_P$ is not regular, therefore ${\rm ht}P\geq2$, thus ${\rm codim}(\overline{J},A)\geq2$. 

$(2)\Rightarrow(1)$ We know that $A$ is C-M because it is the quotient of C-M $R[T_1,\ldots ,T_n]$ 
by an ideal $(h)$ of height one generated by exactly one element, (see \cite[Theorem 18.13]{eisenbud}). 
Therefore, for any prime ideal $P$ of $A$ the local ring $A_P$ is C-M. Then,
\[ {\rm depth}(A_P)={\rm dim}(A_P)\geq {\rm min}(2,{\rm dim}(A_P)). \]
Thus $A$ satisfies the condition (S2) of the Serre's Criterion. Besides, $A$ satisfies condition (R1). In fact, any prime ideal $P$ 
of $A$ of height at most 1 does not contain $\overline{J}$, because ${\rm codim}(\overline{J},A)={\rm ht}_A(\overline{J})\geq2$, or $\overline{J}=A$, hence $P$ is not 
in the singular locus of $A$, that means the regularity of the local ring $A_P$. 

Since, $\overline{J}=A$ if and only if $J=B$ then, for the equivalence between (2) and (3) we can assume that $\overline{J}\subsetneq A$ 
(respectively $J\subsetneq B$).

$(2)\Rightarrow(3).$ Let $P$ be a prime ideal of $B$ that contains $J$, then by hypothesis ${\rm ht}_A(\overline{P})\geq2$. 
Let $\overline{P_0}\subsetneqq \overline{P_1} \subsetneqq \overline{P_2}=\overline{P}$ be a chain of primes in $A$, then one can see the corresponding chain of 
prime ideals in $B$ adding the zero ideal, which is a prime ideal, 
$Q_0=(0)\subsetneqq Q_1=P_0\subsetneqq Q_2=P_1 \subsetneqq Q_3=P_2=P$, that means ${\rm codim}(J,B)\geq3$ 

$(3)\Rightarrow(2)$ Let $P$ be a prime ideal  of $A$ that contains $\overline{J}$, and let $Q$ be the prime ideal of $B$ 
that correspond to $P$. Clearly, $J\subseteq Q$ as subsets of $B$. We know that ${\rm ht}(Q)\geq3$ and $(h)\subseteq Q$, 
therefore $Q$ contains a minimal prime ideal of $(h)$, say $Q_0$, 
which has height one by Krull's Principal Ideal Theorem.
In virtue of this, we know that there exists a saturated chain of primes ideals 
of $B$, \[ Q_0=(0)\subsetneqq Q_1 \subsetneqq Q_2 \subsetneqq Q_3 \subseteq Q,\] since $B$ is a 
catenary ring and ${\rm ht}Q\geq3$, and therefore any saturated chain of prime ideals from $(0)$ to $Q$ has the same length, 
that is, ${\rm ht}(Q)$, which is a least three. Therefore, looking at the corresponding chain in $A$, and denoting by $P_{i-1}$ the prime 
ideal of $A$ corresponding to $Q_i$, we get, starting with the class of $Q_1$, the following saturated chain: $P_0 \subsetneqq P_1 \subsetneqq P_2 \subseteq P$, then ${\rm ht}P\geq2$. In conclusion, 
${\rm codim}(\overline{J},A)\geq2$.
\end{proof}
\begin{remark}\label{codimension}
 An important fact is that for $R=k[x_1,\ldots ,x_r]$, $I$ an ideal of $R$ and $B=R[T_1,...,T_n]$ we know that the ${\rm codim}(I,R)={\rm codim}(IB,B)$, 
because by previous results we get
\[n+r-\codim(IB,B)={\rm dim}(B/IB)={\rm dim}((R/I)[T_1,\ldots,T_n])=\] \[{\rm dim}(R/I)+n={\rm dim}R-{\rm codim}(I,R)+n=n+r-{\rm codim}(I,R).\]

\end{remark}
We want to find necessary and sufficient conditions for the forcing data $f_1,...,f_n$ and $f$ on the base ring of polynomials
$R=k[x_1,...,x_n]$, such that the associated forcing algebra turns out to be a normal domain. The previous lemma gives a condition over $A$ 
and the Jacobian ideal $J$ of the partial derivatives of the forcing equation, which involves, as seen before, again the forcing ideal and 
new forcing equations defined by the partial derivatives of the original forcing data. This suggests that a suitable condition for normality 
over the base $R$ should involve the forcing data and its partial derivatives. The following collection of examples start to give us 
a good first intuition of the phenomenon.
\begin{example}\label{normalintuition}
 Let $k$ be a perfect field and let's define $R=k[x,y]$; $B=k[x,y,T_1,T_2]$; $A=B/(h)$ and
 \[h=x^aT_1+y^bT_2+x^cy^d,\] where $a,b,c$ and $d$ are 
 nonnegative integers. After computations we have that the Jacobian ideal
 \[J=(x^a,y^b,x^cy^d,ax^{a-1}T_1+cx^{c-1}y^d,by^{b-1}T_2+dx^cy^{d-1}).\]
 Let $D\subseteq R$ be the ideal generated by all the partial derivatives of the generators of the forcing ideal 
 $I=(f_1,f_2,f)=(x^a,y^b,x^cy^d)$,
 i.e., \[D=(ax^{a-1},by^{b-1},cx^{c-1}y^d,dx^cy^{d-1}).\] By Lemma \ref{domain}, $A$ is a domain for any nonnegative values of the exponents.
 
 After elementary considerations we see that ${\rm codim}(J,B)\geq3$ or $J=B$ if and only if some of the following seven cases occur:
 
 i) $a=0$.
 
 ii) $a=1$.
 
 iii) $b=0$.
 
 iv) $b=1$.
 
 v) $d=c=0$.
 
 vi) $c=1$ and $d=0$.
 
 vii) $c=0$ and $d=1$.
 
 In fact, in any other case $J\subseteq (x,y)B$, and 
 therefore ${\rm codim}(J,B)\leq 2$. Moreover, it is also elementary to see that the previous seven cases are exactly the ones
 in which the ideal $I+D$ is equal to $R$.
 
 In conclusion, in virtue of the previous Lemma, $A$ is a normal domain if and only if $I+D=R$.
\end{example}
\begin{remark}
Suppose that $k$ is an algebraically closed field. Continuing with the notation of the former example, let's write
$V=V(I)\subseteq k^2$, $W=V(D)\subseteq k^2$, $Y=V(h)\subseteq k^4$ and $S=V(J)\subseteq k^4$ denote the corresponding affine varieties and
 $\pi:S\rightarrow V$ the natural projection to the first two coordinates. Geometrically, 
Example \ref{normalintuition} suggests that the normality of the variety $X$ (which is equivalent to the normality of the forcing algebra,
see \cite[Exercise I.3.17]{hartshornealgebraic}), is related to the intersection of $V$ and $W$, because $V\cap W=\emptyset$, if and only
if $I+D=R$. In fact, this is true for arbitrary polynomial data $f_1,f_2$ and $f\in R$ as we will see. 
  
First, by Lemma \ref{domain}, $A$ is an integral domain if and only if ${\rm ht}I\geq2$ or $I=R$. So, let's assume that $A$ is a domain and
$I\subsetneq R$, otherwise $V=\emptyset$ and $J=B$, being $A$ normal, by Lemma \ref{sing}. Thus, ${\rm ht}I\geq2$, which means that the minimal prime ideals over $I$
are just finitely many maximal ideals, since $\dim R=2$. But, by the Nullstellensatz (see \cite[Exercise 7.14]{atimac}) this points correspond exactly to the points of $V$. Therefore, let's write 
$V=\{v_1,...,v_r \}$. 
   
Moreover, let $S$ be the singular locus of $Y$	 in the sense that, if we consider $S$ as a subvariety of $Y$.
By previous comments $S$ is the finite union of its (singular) fiber varieties $S_{v_i}=\pi^{-1}(v_i)$. 
Now, by Lemma \ref{sing}, $Y$ is a normal variety if and only if ${\rm codim}(S,K^4)\geq3$
 (which is equivalent to ${\rm codim}(S,Y)\geq 2$).
 
 Assume, that $V\cap W\neq\emptyset$, i.e., $I+D\subsetneq R$, and let's prove that $Y$ is not normal. In fact, we know that $J\subseteq (I+D)B$. Therefore, by Remark 
 \ref{codimension} \[\codim (S,k^4)=\codim (J,B)\leq \codim((I+D)B,B)=\codim (I+D,R)\leq2,\]
 implying that $Y$ is not normal.
 
 Conversely, assume that $V\cap W=\emptyset$. Then, for any point $v\in V$, there exists some $\partial f_i(v)/ \partial x_j\neq0$,
  because if not all the partial derivatives of the forcing data would be zero at $v$ (the elements $\partial f(v)/ \partial x_j$ are also zero, 
  because we can write them as a linear combinations of the $\partial f_i(v)/ \partial x_j$, see Remark \ref{derivatives}), 
  implying that $v\in W$, but that is impossible. 
  
  Clearly, $S_{v}=V(G)$, where $v=(a,b)\in k^2$ and 
 \[G=(x-a,y-b,\partial f_1(v)/ \partial xT_1+\partial f_2(v)/ \partial xT_2+\partial f(v)/ \partial x,\]
  
\[\partial f_1(v)/ \partial yT_1+\partial f_2(v)/ \partial yT_2+\partial f(v)/ \partial y).\]
 But, under the condition that some $\partial f_i(v)/ \partial x_j\neq0$, it is elementary to see that $\codim (G,B)\geq3$. In conclusion, 
 $\codim (S_{v}, k^4)\geq3$, implying that $\codim (S,k^4)$, being the minimum of the codimension of its singular fibers, is bigger or equal than three, which means 
 the normality of $Y$. 
 \end{remark}

 Besides, if we move to the next dimension, i.e., $R=k[x_1,x_2,x_3]$ and $B=R[T_1,T_2,T_3]$, then, it is possible to see in a natural way 
 that a necessary condition for the normality of $Y$ is that $(\dim V\cap W)<1$ (here we assume that the dimension of the empty set is $-1$).
 Because, suppose by contradiction that $\dim V\cap W\geq1$. For any point $v\in V\cap W$, by Remark \ref{derivatives} 
 and Lemma \ref{forcingfiber} the fiber $S_v\cong \mathbb{A}^3_k$. Therefore, $(V\cap W)\times \mathbb{A}^3_k\subseteq S$.
 But, $\dim (V\cap W)\times \mathbb{A}^3_k\geq1+3=4$, and so, $\dim S\geq4$, thus, $\codim (S,k^6)\leq2$, implying that $Y$ is not normal. 
 Note that this argument works independent from the number of variables. However, this case was very suitable to obtain the right 
 intuition about the desired condition i.e., $\dim(V\cap W)<r-2$.
 
 Heuristically, one can compute the dimension of $S$ by knowing the general behavior of the dimension of the fibers $S_v$ and the dimension
of the base space $V$. Now, by Lemma \ref{forcingfiber} the fibers $S_v$ have maximal dimension exactly when the rank of the forcing
matrix is minimal, i.e., when the point $v$ belongs to $W\cap V$. Therefore, to guarantee that the dimension of $Y$ is not so big
(in order to maintain the codimension big enough), we need to bound the dimension of the subvariety of $V$ with maximal dimensional singular
fibers, i.e., the dimension of $V\cap W$. In fact, assuming that $Y$ is irreducible, the right necessary and sufficient condition for 
$Y$ being an (irreducible) normal variety is that ($\dim V\leq r-2$ and) $\dim V\cap W\leq r-3$, where $V,W\subseteq k^r$.

First, in order to get a better intuition about the fibers, the following proposition tells us that the points of $\Spec R$ with
fibers completely singular are exactly the points of $V(I)\cap V(D)$.

\begin{proposition}
 Let $R=k[x_1,\ldots ,x_r]$ be the ring of polynomials over a perfect field $k$; $B=R[T_1, \ldots ,T_n]$; 
 $h=f_1T_1+\cdots+f_nT_n+f$; $f,f_1, \ldots , f_n \in R$; $A=B/(h)$; 
 $I=(f,f_1, \ldots , f_n)$; $D=(\partial f/ \partial x_j, \partial f_i/ \partial x_j)$ and 
\[J:=(h,(\partial h/\partial x_j),(\partial h/ \partial T_i)).\]
 Let $\varphi: Y=\Spec A\rightarrow X=\Spec R$ be the forcing morphism.
 Choose a point $x\in Y$ with nonempty fiber $\varphi^{-1}(x)$.
 Then $x\in X$ has fiber completely singular i.e., $\varphi^{-1}(x)\subseteq V(J)\in Y$ if and only if $x\in V(I+D)\subseteq X$.
\end{proposition}
\begin{proof}

We know from the Corollary of the Jacobian Criterion that for any prime ideal $y\in Y$, $A_y$ is not regular if and only if $y\in V(J)$. 
Let $x\in V(I+D)$ and $Q\in \varphi^{-1}(x)$. Then, $(I+D)B\in Q$ and so $J\in Q$, meaning that $Q\in V(J)$.

Conversely, let's consider a point $x\in X$, such that $\varphi^{-1}(x)\subseteq V(J)$. Now, it is elementary to see that the last condition means 
that $\varphi^{-1}(x)=V(J_x)$, where 
\[\varphi^{-1}(x)= A=k(x)[T_1, \ldots, T_n] /(f_1(x)T_1 + \ldots + f_n(x)T_n+f(x)),\]
and $J_x=(\sum_{i=1}^n(\partial f_i(x)/\partial x_j)T_i +(\partial f(x)/\partial x_j),$
for $i,j\in \{1,...,n\}$. 

Firstly, if $f_i\notin x$, for some $i$, then the fiber $\varphi^{-1}(x)$ is completely regular, because, by previous comments (Ch. 1 \S2) 
$\varphi^{-1}(x)\cong\mathbb{A}^{n-1}_{k(x)}$.

Secondly, if $f\notin x$, then $f(x)=0$. But, we know that $f_1(x)=\cdots=f_n(x)=0$, therefore the fiber is empty, since $h=f(x)\neq0\in k(x)$.
But, it contradicts our hypothesis.
Note that, until now, we know that $h=f_1(x)T_1 + \ldots + f_n(x)T_n+f(x)=0$.

Thirdly, suppose that $\partial f_i/\partial f_j\notin x$, for some $i,j\in \{1,...,n\}$, that means, $\partial f_i(x)/\partial f_j=0$.
We consider two cases: Suppose that 
$\partial f(x)/\partial f_j\neq0$. Then, since $h=0$, the ideal $Q=(T_1,...,T_n)\in \varphi^{-1}(x)$, but 
\[\sum_{i=1}^n(\partial f_i(x)/\partial x_j)T_i +(\partial f(x)/\partial x_j)\notin Q.\]
Therefore $Q\notin V(J_x)$, a contradiction.
In the second case, i.e., $\partial f(x)/\partial f_j=0,$, the prime ideal $Q'=(T_1,...,T_i-1,...,T_n)\in\varphi^{-1}(x)$, but 
\[\sum_{i=1}^n(\partial f_i(x)/\partial x_j)T_i +(\partial f(x)/\partial x_j)=\]
\[\sum_{i=1}^n(\partial f_i(x)/\partial x_j)T_i\notin Q'.\]
So, again, $Q'\notin V(J_x)$, a contradiction.

Lastly, if $\partial f(x)/\partial x_j\neq0$, for some $j$, then, due to the last results 
\[\sum_{i=1}^n(\partial f_i(x)/\partial x_j)T_i +(\partial f(x)/\partial x_j)=\partial f(x)/\partial x_j\in J_x,\]
thus $\varphi^{-1}(x)=V(J_x)=\emptyset$. But, this is not possible, because the fiber is not empty. 

In conclusion, $\varphi^{-1}(x)\subseteq V(J)\in Y$, as desired.
\end{proof}

 Now, we present the statement of the normality criterion for forcing algebras over the ring of polynomials with coefficients in a perfect 
 field.

\begin{theorem}\label{normalcriterion}
 Let $R=k[x_1,\ldots ,x_r]$ be the ring of polynomials over a perfect field $k$; $B=R[T_1, \ldots ,T_n]$; $f,f_1, \ldots , f_n \in R$;
 $I=(f,f_1, \ldots , f_n)$;$D=(\partial f/ \partial x_j, \partial f_i/ \partial x_j)$, for $i,j\in\{1,...,n\}$. Then, the forcing algebra for this data $A$ is 
 a normal domain if and only if the following two conditions hold: 
\begin{enumerate}[(a)]
 \item ${\rm codim}(I,R)\geq2$, or $I=R$.
\item ${\rm codim}(I+D,R)>2$, or $I+D=R$.
\end{enumerate}
Moreover, in the case that all $f_i=0$, then (b) is a necessary and sufficient condition for $A$ being a normal ring.

\end{theorem}
 \begin{proof}
  We have already proved in Lemma \ref{domain} that (a) is a necessary and sufficient condition for $A$ being an integral domain. 
  Let's prove that (b) is equivalent to normality. 
Effectively, following Lemma \ref{sing} we just need to see the condition (b) is equivalent to ${\rm codim}(J,B)>2$, or $J=B$. Let's denote the 
last condition by (b').
By Remark \ref{derivatives} we know that $J\subseteq (I+D)B$. Suppose that (b') holds. First, if $J=B$, then $(I+D)B=B$, implying 
$I+D=R$. Second, if ${\rm codim}(J,B)>2$, then by Remark \ref{codimension} we get

\[{\rm codim}(I+D,R)={\rm codim}((I+D)B,B\geq {\rm codim}(J,B)>2.\]
Conversely, assume that (b) holds and $J\neq B$. We prove that ${\rm codim}=(J,B)>2$. 
Let $Q$ be a prime ideal of $B$ that contains $J$. First, assume that $(I+D)B\subseteq Q$, then $I+D\neq R$, therefore
${\rm codim}(I+D,R)>2$, so, again by Remark \ref{codimension}  
${\rm codim}((I+D)B,B)>2$, which implies that ${\rm codim}(Q,B)>2$. Second, suppose that $I+D\nsubseteq Q$, then necessarily one of the partial 
derivatives $\partial f/ \partial x_j$ or $ \partial f_i/\partial x_j$ is not contained in $Q$, because $IB\subseteq J\subseteq Q$. 
In fact, there exits some $b\in {1, \ldots, n}$ and some $c\in {1, \ldots, r}$ with $\partial f_b/\partial x_d\notin Q$, cause if 
not, all $\partial f_i/ \partial x_j$ would be contained in $Q$ and also the elements 
$\sum_{i=1}^n(\partial f_i/\partial x_j)T_i+\partial f/\partial x_j$ and therefore $\partial f/\partial x_j$ for any $j$, thus 
$D$ would be also contained in $J$, which is not the case. For simplicity suppose that $Q$ not contained the element 
$\alpha:= \partial f_1/\partial x_1$ and let's write $l:=\sum_{i=1}^n(\partial f_i/\partial x_1)T_i+\partial f/\partial x_1$. 
Let $\psi$ be the following homomorphism of $R_{(\alpha)}$ algebras 
\[\psi:B_{(\alpha)}\approxeq R_{(\alpha)}[T_1,\ldots ,T_n]\longrightarrow R_{(\alpha)}[T_2, \ldots, T_n], \]

that sends $T_1$ to $g:=-\alpha^{-1}(\sum_{i=2}^n(\partial f_i/\partial x_1)+\partial f/\partial x_1)$ and $T_j$ to $T_j$, for 
$j\geq 2$. Clearly, $\psi$ is surjective. Moreover, $ker(\psi)=(T_1-g)$. To see this let $S\in ker(\psi)$. Then using the binomial 
expansion we can write it in the form:
\[ S=S(x_1,\ldots, x_r,(T_1-g)+g,\ldots,T_n)=S_0(x_1,\ldots, x_r,(T_1-g),\ldots,T_n)+\] \[S(x_1,\ldots, x_r,g,\ldots,T_n),\]
\[=S_0(x_1,\ldots, x_r,(T_1-g),\ldots,T_n)+\psi(S)\]
\[=S_0(x_1,\ldots, x_r,(T_1-g),\ldots,T_n),\]
with $S_0$ being divisible by $T_1$, which implies that the former expression is divisible by $T_1-g$. Thus $S\in (T_1-g)$.

On the other hand, in the ring $ R_{(\alpha)}[T_1,\ldots,T_n]$ we know that $(T_1-g)=(l)$, therefore $\psi$ induces an isomorphism 
between $ R_{(\alpha)}[T_1,\ldots,T_n]/(l)$ and $ R_{(\alpha)}[T_2,\ldots,T_n]$. Denote by $Q_0$ the image under $\psi$ of 
$QR_{(\alpha)}[T_1,\ldots,T_n]$, and assume for the sake of contradiction that ${\rm codim}(Q,B)\leq2$ then we have the following chain of 
inequalities:
\[d:={\rm dim}(B/Q)={\rm dim}B-{\rm codim}(Q,B)=n+r-{\rm codim}(Q,B)\geq n+r-2.\]

Besides, $B$ is a Jacobson ring, hence there exists a maximal ideal $m$ containing $Q$ such $\alpha \notin m$, otherwise $\alpha$ would 
be contained in the intersection of all the maximal ideals containing $Q$, which is $Q$, absurd. Now, let's consider a saturated 
chain of primes ideals from $Q$ to $m$, which exits in virtue of Zorn's lemma. Besides, this chain has length exactly $d$ because $B/Q$ is 
an affine domain and therefore, $d$ is the length of any saturated chain of primes on it (see fundamental results on Chapter 1). Then,
\[ Q=Q_0\subsetneqq Q_1 \subsetneqq, \ldots, \subsetneqq Q_{d-1} \subsetneqq Q_d=m.\]

Now, we can consider this chain in $R_{(\alpha)}[T_1,\ldots,T_n]$, because no $Q_i$ contains $\alpha$. This shows that 
${\rm dim}(R_{(\alpha)}[T_1,\ldots,T_n])/Q^e\geq d$ and, in fact, the equality holds because we are localizing and thus the dimension 
cannot be bigger that the dimension of the original ring. Besides, $\psi$ induces an isomorphism between 
$R_{(\alpha)}[T_1,\ldots,T_n]/Q^e$ and $R_{(\alpha)}[T_2,\ldots,T_n]/Q_0$, then finally, recalling that ${\rm codim}(I,R)\geq2$ and 
that
$l\in Q$ we get 
\[d={\rm dim}(R_{(\alpha)}[T_1,\ldots,T_n])/Q^e)={\rm dim}(R_{(\alpha)}[T_2,\ldots,T_n])/Q_0)\leq \]
\[ {\rm dim}(R_{(\alpha)}[T_2,\ldots,T_n])/I^e) \leq {\rm dim}(R[T_2,\ldots,T_n])/I^e)\]
\[={\rm dim}((R/I)[T_2,\ldots, T_n])={\rm dim}(R/I)+n-1=\] 
\[{\rm dim}R-{\rm codim}(I,R)+n-1\leq r+n-1-2< n+r-2.\]
Which is a contradiction with the former estimate of $d$.
Finally, if all $f_i=0$ then $J=I+D$ and then from the fact that ${\rm codim}((I+D),R)={\rm codim}((I+D),B)$ we deduce from 
Lemma \ref{sing} 
that condition (b) is equivalent to the normality of $A$. 

 \end{proof}
 Now, we state a direct application of the previous Theorem to normal affine varieties. 
 As said before, our convention is that $\dim\emptyset=-1$.
\begin{corollary}Let $R=k[x_1,\ldots ,x_r]$ be the ring of polynomials over an algebraically closed field $k$; 
$B=R[T_1, \ldots ,T_n]$; $f,f_1, \ldots , f_n \in R$; \[I=(f,f_1, \ldots , f_n)\] and 
$D=(\partial f/ \partial x_j, \partial f_i/ \partial x_j)$. 
Assume that $(h)$ is a radical ideal, where $h=f_1T_1+\cdots+f_nT_n+f.$
 Let's denote by $V=V(I)\subseteq k^r$ and $W=V(D)\subseteq k^r$ the affine varieties defined by $I$ and $D$, respectively. 
 Then, $X=V(H)\subseteq k^{n+r}$ is a normal (irreducible) variety if and only if the following two conditions holds simultaneously
\begin{enumerate}[(1)]
 \item ${\rm dim}V\leq r-2$.
\item ${\rm dim}(V\cap W)<r-2$.

Moreover, in the case that all $f_i=0$, then  (2) is a necessary and sufficient condition for $X$ being a normal (irreducible) variety.
\end{enumerate}
\end{corollary}
\begin{proof} Recall that a variety is normal if for any point $x\in X$, the stalk $\mathcal{O}_{X,x}$ is a normal domain 
(see \cite[Exercise I.3.17]{hartshornealgebraic}).
Since $(h)$ is a radical ideal, we know that the forcing algebra $A=B/(h)$ is exactly the ring of coordinates of $X$. Since $X$ is 
affine and normality is a local property we have that $X$ is a normal (irreducible) variety if and only if $A$ is a normal domain.
Besides, from Hilbert's Nullstellensatz we get 
\[{\rm dim}V={\rm dim}(R/I(V))={\rm dim}(R/\rad(I))={\rm dim}(R/I)\] \[={\rm dim}R-{\rm codim}(I,R)=r-{\rm codim}(I,R),\] 
and analogously \[{\rm dim}(V\cap W)=r-{\rm codim}(I+D,R).\]
 From this and the fact that $V=\emptyset$ (or $V\cap W=\emptyset$), if and only if $I=R$ (or $I+D=B$), we rewrite the 
 conditions (a) and (b) of the former theorem as (1) and (2).
\end{proof}

As a comment, we say that the discussion beginning at Example \ref{normalintuition} is essentially the way in which the above 
criterion of normality was discovered.
Lastly, in order to support the former intuition we dedicate the next pair of sections to study two interesting and enlightening examples.

\section{An Enlightening Example}
In this section we study an specific example of a forcing algebra with several forcing equations and we explore the some interesting properties. This example shows how rich and interesting is the formal study of forcing algebras on its own.

 Let $R=k[x,y]$ be the ring of polynomials over a (perfect) field $k$, $B=R[T_1,T_2]$, $A=B/H$, where 

\[ H=(h_1,h_2)=(xT_1+yT_2,yT_1+xT_2)=\left( \left( \begin{array}{cc}x&y\\y&x\end{array}\right) \cdot \left( \begin{array}{c}T_1\\T_2 \end{array}\right)\right). \]

The determinant of the associated matrix $M$ is $x^2-y^2=(x+y)(x-y)$. It is easy to check that $h_1$ is irreducible and that $h_2$ 
does not belong to the ideal generated by $h_1$. Therefore $h_1,h_2\subseteq B$ is a regular sequence and hence, by former comments
$H$ has pure codimension $2$.

Let $P$ be a minimal prime of $H$. Then, by a previous remark, $P$ contains the elements $\det MT_i=(x-y)(x+y)T_i$ for $i=1,2$. If 
$\det M\notin P$, then $T_i\in P$, and therefore $P=(T_1,T_2)$. Now, assume that $\det M\in P$, then $x-y\in P$ or $x+y\in P$. In the 
first case, $h_1-T_1(x-y)=y(T_1+T_2)$ should be in $P$. But, if $y\in P$ then $x=(x-y)+y\in P$, which implies that $P=(x,y)$. 
If $T_1+T_2\in P$ then it is easy to check that $P=(x-y,T_1+T_2)$, since this is a prime ideal containing $H$. On the other hand, if $x+y\in P$, 
then, similarly we see that $P=(x,y)$, or $P=(x+y,T_1-T_2)$. In conclusion, the minimal primes of $H$ (which are, in fact, the 
associated primes of $H$, because $A$ is a Cohen-Macaulay ring) are the four ideals $P_1=(T_1,T_2)$, $P_2=(x,y)$, $P_3=(x-y,T_1+T_2)$ and 
$P_4=(x+y,T_1-T_2)$. 

This example shows that Theorem \ref{irreducibility} is false for several forcing equations, since $\Spec A$ is not an irreducible space
 but the ideal generated by the forcing data $(x,y)$ has height two.
 
 %On the other hand, $\rad H$ is the intersection of the four minimal primes found above, thus the element 
 %$xT_1-yT_2=T_1(x+y)-y(T_1+T_2)\in \rad H$, but it is elementary to see that $xT_1-yT_2\notin H$. So, $H$ is not a radical and thus $A$ is
  %not reduced. It implies, in particular, that $A$ is not a normal ring (i.e. the direct product of normal domains), due to the fact that 
   %these kind of rings are clearly reduced. Later, we see in a geometrical way the non-normality of $A$.

Let $V_i=V(P_i)\subseteq k^4$ be the affine variety define by $P_i$, which correspond to the irreducible 
components of $V=V(H)$. Now, the intersections of any couple of this components correspond to singular points of $V$ (we assume for a while that 
$k$ is algebraically closed, and we replace $H$ by $\rad H$ in order to work with the corresponding variety $V$), because 
the ring of coordinates of $V$ localized at the maximal ideal corresponding to such a points has at least two irreducible components 
and therefore it is not an integral domain, in particular, it is not a regular local ring, since local regular rings are domains.

This is a way to see geometrically the non-normality of $V$, because the normality is a local property and the localization at these 
intersection points, say $\idealp\in{\rm Spm}(A)$, is not a normal ring. 
In fact, a local ring has clearly a connected spectrum, therefore $A_{\idealp}$ cannot be a direct
 product of normal domains (\cite{brennergomezconnected} \S1). Besides, by the former comment, $A_{\idealp}$ cannot be neither a normal domain.
 
 Returning to our computations, we see that the intersection of these irreducible components are, in general, defined by lines and,
 in two cases, defined by just one point. In fact, 
$V_1\cap V_2=V(x,y,T_1,T_2);V_1\cap V_3=V(T_1,T_2,x-y);V_1\cap V_4=V(T_1,T_2,x+y)$;$V_2\cap V_3=V(x,y,T_1+T_2);V_2\cap V_4=V(x,y,T_1-T_2)$
 and $V_3\cap V_4=V(x,y,T_1,T_2)$.

Furthermore, It is easy to see that $\Spec A$ is connected (see \cite[Proposition 1.2.]{brennergomezconnected}), since we are in the homogeneous case.
Moreover, $V(P_1)$ is an horizontal component, $V(P_2)$ a vertical component and
 $V(P_3)$ and $V(P_4)$ behave like ``mixed'' components i.e., they do not dominate the base nor are they the preimage of a subset of the base.
 Besides, $\Spec A$ is also locally (over the base) connected because every pair of 
 minimal components have non-empty intersection and the elementary fact that the minimal primes of a localization are exactly the 
 minimal primes of the original ring not intersecting the multiplicative system.
 
In the case that $k$ is a perfect field, we can use also the Jacobian Criterion in order to prove again that $A$ is not a normal ring. 
In fact, as seen before, the pure codimension of $H$ is two, since $\left\lbrace h_1,h_2 \right\rbrace$ is a regular sequence. So, 
the singular locus in $\Spec A$ is given by the $2\times2$ minors of the Jacobian matrix defined by the partial derivatives of the $h_i$, 
that is, 
\[J=\left( T_1^2-T_1^2,x^2-y^2,yT_1-xT_2,xT_1-yT_2\right).\] 
Thus in order to test normality we should find the codimension 
of $J$ in $A$ and determine if it is bigger or equal than two. Since the pure codimension of $H$ is two we can translate our problem 
to the ring of polynomial in four variables $B=k[x,y,T_1,T_2]$ and to test if the corresponding Jacobian ideal 
\[J_0=\left( T_1^2-T_1^2,x^2-y^2,yT_1-xT_2,xT_1-yT_2,h_1,h_2\right)\] has codimension bigger or equal to four (in general, the codimension
 of a prime ideal decreases in $n$, if we mod out by ideals of pure codimension $n$, mainly because an affine domain is catenary and its
 dimension is the length of any maximal chain of prime ideals \cite[Corollary 13.6]{eisenbud}.
But, after some computations we can show that the prime ideals that contain $J_0$ are exactly the ideals defining the varieties 
corresponding to the intersections of pairs of the irreducibles components of $V$. That is,
$(x,y,T_1,T_2)$, $(T_1,T_2,x-y)$, $(T_1,T_2,x+y)$, $(x,y,T_1+T_2)$ and $(x,y,T_1-T_2)$. 
Therefore $\codim(J_0,B)=3$, and then, $\codim(\overline{J_0},A)=1<2$, implying that $A$ does not satisties Serre's condition (R1). Hence, 
by Serre's Normality Criterion $A$ is not a normal ring. Moreover, by the same reason $B/\rad H$ is not a normal ring, and this is 
equivalent to the non-normality of the variety $V(H)\subseteq k^4$. 

Geometrically, if $k$ is an algebraic closed field, it means just that the singular points of $V$, which correspond to the maximal 
ideals containing $J_0$, are exactly the points in the intersections of the different irreducible components of the variety,
which correspond to the geometrical intuition of singularities.

This example suggests the following conjecture.
\begin{conjecture}
 In the homogeneous case, assume that $R=k[x_1,...,x_r]$, and suppose  $H=(h_1,...,h_m)=P_1\cap ... \cap P_s$, where $P_i$ are the minimal primes, for $i=1,...,s$. 
 Then $V(P_i)\cap V(T_1,...,T_n)\neq\emptyset$.
\end{conjecture}

\section{An Example of Normalization}
On this section we will compute explicitly the normalization of a forcing algebra by elementary methods illustrating how good examples
lead us in a natural way to the study of general basic properties of normal domains.
 
Let $k$ be a perfect field. Our example is a particular case of the Example \ref{normalintuition}. Let $R=k[x,y]$, $B=R[t,s]$, $A=B/(h)$,
where $h=x^2t+y^2s+xy$. Now, with the notation of section 3, $I=(x^2,y^2,xy$), $D=(x,y)$, and so, $I+D=(x,y)$. By Theorem \ref{normalcriterion} $A$ is a 
non-normal domain, because $\codim(I,B)\geq2$, but $\codim(I+D,B)=2$. Besides, the integral closure, or normalization of $A$, $\overline{A}$ is a 
module-finite extension of $A$ (in general, this is true for finitely generated algebras over complete local rings, see 
\cite[Exercise 9.8]{hunekeswanson}).

Now, we will give an explicit description of $\overline{A}$ as an affine domain. 

First, let $K=K(A)$ be the field of fractions of $A$ and let $u=tx/y\in K$. Then, if we consider the forcing equation $h$ in $K[t,s]$, we 
get the following integral equation for $u$, after multiplication by $t/y^2$: 
\[ (tx/y)^2+(tx/y)+st=0. \]
Let $A'=A[u]$ be the $A-$subalgebra of $K$ generated by $u$. So, we rewrite $h$ considered in $A'$, by means of $yu=xt$, to obtain 
the equation $0=h=y(xu+ys+x)$. But, $y\neq0$, therefore $xu+ys+x=0$. 

Let $C=k[X,Y,T,S,U]$ be the ring of polynomials. Define $\phi:C\rightarrow A'$ the homomorphism of $k-$algebras sending each capital
 variable into its corresponding small variable. Note that from the previous considerations the ideal
 $P=(YU-XT,XU+YS+X,U^2+U+TS)\subseteq \ker\phi$. We will see that $P=\ker\phi$. Effectively, let's write $E=k[X,Y,U,T]/(YU-XT)$. Then, 
 $E$ is a forcing algebra and by Theorem \ref{normalcriterion} is a normal domain. 
 
First, we prove that $P$ is a prime ideal. Define $Q=K(E)$, then, informally if we consider the equations
\[XU+YS+X=U^2+U+TS=0\]
in the variable $S$ and solve them, it lead us to obtain the equality $S=-(U^2+U)/T=-(XU+X)/Y$ in a ``suitable'' field of fractions. But,
in fact, it hods that 
\[-(U^2+U)/T=-(XU+X)/Y\in Q,\]
because 
\[-Y(U^2+U)=-TXU-XT=-T(XU+X)\in D\],
 due to the fact that $YU=XT\in E$. Write $S'=-(U^2+U)/T=-(XU+X)/Y\in Q$ and consider the natural homomorphism $\psi: E[S]\rightarrow E[S']\subseteq Q$,
  where $E[S]$ denote the ring of polynomials in the variable $S$. We will prove that $\ker \psi=(XU+YS+X,U^2+U+TS)$. For that we need  
  the following basic lemma about normal domains:

  \begin{lemma}\label{denominator}
   Let $R$ be a normal domain, $q\in K(R)$, 
   \[I=(bx-a\in R[x]: q=a/b; a,b\in R\},\] and 
   $(R:q)=\{b\in R:bq\in R\}$ be the denominator ideal.
   Consider the homomorphism of $R-$algebras
   \[\varphi:R[x]\rightarrow R[q]\subseteq K(R),\]
   sending $x$ to $q$. 
   Then the following holds:
   \begin{enumerate}
    \item If $q\notin R$, then $\codim((R:q),R)=1$.
    \item Suppose that $(R:q)=(b_1,...,b_m)\in R$, such that $q=a_i/b_i$, for 
    some $a_i\in R$. Then, $I=(b_1x-a_1,...,b_mx-a_m)$.
    \item $\ker \varphi=I$.
    
   \end{enumerate}

  \end{lemma}
\begin{proof}
(1) It is a well known fact that any normal Noetherian domain is the intersection of its localizations on primes of height one (see 
\cite[Corollary 11.4]{eisenbud}). We argue by contradiction. If $\codim((R:q),R)\geq2$, then $(R:q)$ is 
not contained in any prime ideal $P\subseteq R$ of height one. In particular, there exists for every such prime ideal $P$ an element 
$b_P\notin P$, but $b_P\in (R:q)$, meaning that there is $a_P\in R$, with $q=a_P/b_P\in R_P$. In conclusion, $q\in \cap_{\text{ht}P=1}R_P=R$.

(2) Let $bx-a\in I$. That means, in particular, that $b\in (R:q)$. So, we can write $b=c_1b_1+\cdots+c_rb_r\in R$, for some 
$c_i\in R$, $i=1,...,r$. Now, let $a_i\in R$ be elements such that $q=a_i/b_i$. Since, 
\[a=bq=\sum_{i=1}^nc_ib_iq=\sum_{i=1}^nc_ia_i,\]
it is
 straightforward to verify $bx-a=\sum_{i=1}^rc_i(b_ix-a_i)$, as desired.
 
(3) Clearly $I\subseteq \ker \varphi$. For the other containment, let $f\in\ker\varphi$ we argue by induction on the degree of $f$.
Write $f=v_nx^n+\cdots+v_0$. The case $n=1$ is clear. So, assume $n\geq2$. First, we know that 
\[v_nq^n+\cdots+v_0=0\in K(R),\]
then after multiplying by $v_n^{n-1}$, we get the integrity equation for $v_nq$, 
\[(v_nq)^n+v_{n-1}v_n(v_nq)^{n-1}+\cdots+v_0v_n^{n-1}=0.\]
So, 
 $v_nq\in R$, because $R$ is a normal domain. Therefore, there exists $d\in R$ such that $q=d/v_n$. Now, $f-x^{n-1}(v_nx-d)\in\ker\varphi$, 
 and it has lower degree. Thus, by the induction hypothesis $f-x^{n-1}(v_nx-d)\in I$, and then $f\in I$, because $v_nx-d\in I$.
\end{proof}

We continue with our discussion, by abuse of notation we write with the same capital letters its classes in $E$. Now, we know that 
$Y,T\in (E:S')$. Besides, $(X,T)\in E$ in a prime ideal of codimension one in $E$, therefore in virtue of Lemma \ref{denominator}(1), 
$(Y,T)=(E:S')$. Hence, applying again Lemma \ref{denominator}(2)-(3) we see that
\[\ker\psi=((Y)S+(XU+X),(T)S+(U^2+U)),\]
as desired.
In conclusion,
\[E[S]/(XU+YS+X,U^2+U+TS)\cong E[S']\]
is an integral domain, therefore 
\[C/P\cong E[S]/(XU+YS+X,U^2+U+TS)\]
so is.

On the other hand, since the extension $A\rightarrow A'$ is integral, both rings have the same dimension (it is a direct consequence from the 
Going Up, see \cite[Proposition 4.15]{eisenbud}). But, $\dim A=\dim B-{\rm ht}(h)=3$, and then 
\[3=\dim A'=\dim C/\ker\phi=5-{\rm ht}(\ker\phi),\]
implying ${\rm ht}(\ker\phi)=2$. Besides, it is easy to check that $P\subseteq \ker\phi$ is 
a (prime) ideal of height strictly bigger that one, therefore both ideals coincide. Finally, we can apply Corollary \ref{corjacobian} to 
the affine domain $C/P$. After computations we verify that
\[(U+1)(2U+1),U(2U+1),U(U+1)+ST,ST\in J,\]
where $J$ denotes the Jacobian ideal, defining 
the singular locus of $C/P$. But, easily we check that 
\[C=((U+1)(2U+1),U(2U+1),U(U+1)+ST,ST),\]
therefore the singular locus is empty and
 then $C/P$ is regular, and in particular, normal. 
 In conclusion, an explicit description of the normalization of $A$ as an affine ring is 
\[\overline{A}\cong k[X,Y,T,S,U]/(YU-XT,XU+YS+X,U^2+U+TS).\]

\begin{remark}
One can go forward in a natural way by computing the normalization for forcing algebras with forcing equations of the form $h=x^nt+y^n+xy$, for $n\geq2$. However,
just for the case $n=3$, new methods seem to be needed. In particular, we get an ideal 
\[P=(YU-X^2T,XU+X+Y^2S,U^2+U+XYST).\]
But,
in order to apply Lemma \ref{denominator}, the most challenging part appears to be finding an explicit description of the 
generators of the corresponding denominator ideal, cause \[S'=-(X+UX)/Y^2=-(U^2+U)/XYT,\] and therefore we just know that 
$Y^2,XYT\in(D:S')$, where 
\[E=k[X,Y,U,T]/(YU-X^2T).\]
But, on this case the ideal $(Y^2,XYT)$ is not prime as in the argument before where we get the prime ideal $(X,Y)$ 
as denominator ideal.
\end{remark}
This section suggests on its own a way for forthcoming research on computing the normalization of forcing algebras. 
\section*{Acknowledgement}
Danny de Jes\'us G\'omez-Ram\'irez would like to thank to his parents Jos\'e Omar G\'omez Torres and Luz Stella Ram\'irez Correa for all the assistance, love and inspiration. Besides, he wishes to thank the German Academic Exchange Service (DAAD) for the financial and academic support.

\bibliographystyle{amsplain}
%\bibliography{bibliothek}

\providecommand{\bysame}{\leavevmode\hbox to3em{\hrulefill}\thinspace}
\providecommand{\MR}{\relax\ifhmode\unskip\space\fi MR }
\providecommand{\MRhref}[2]{%
  \href{http://www.ams.org/mathscinet-getitem?mr=#1}{#2}
}
\providecommand{\href}[2]{#2}

\end{document}